\documentclass[a4paper,leqno]{amsart}

\usepackage{latexsym}
\usepackage[english]{babel}
\usepackage{fancyhdr}
\usepackage[mathscr]{eucal}
\usepackage{amsmath}
\usepackage{mathrsfs}
\usepackage{amsthm}
\usepackage{amsfonts}
\usepackage{amssymb}
\usepackage{amscd}
\usepackage{bbm}
\usepackage{subcaption}
\usepackage{latexsym}
\usepackage{color}

\usepackage{pifont}
\usepackage{booktabs} % for '\midrule' macro

\newcommand{\ud}{\mathrm{d}}

\newcommand{\cH}{\mathcal{H}}
\newcommand{\ran}{\mathrm{ran}\,}

\newcommand{\C}{\mathbb C}

\newcommand{\N}{\mathbb N}

\newcommand{\1}{\mathbbm{1}}
\newcommand{\0}{\mathbb{O}}

\newcommand{\mathspan}[1]{\mathrm{span}\left\{#1 \right\}}
\newcommand{\Bop}{\mathscr{B}(\cH)}

\newcommand{\specmeas}[1]{E^{(#1)}}
\newcommand{\mug}[2]{\mu_{#2}^{(#1)}}
\newcommand{\K}[2]{\mathcal{K}\left(#1,#2\right)}

\newcommand{\Kc}[2]{\overline{\K{#1}{#2}}}
\newcommand{\scalar}[2]{\left\langle #1, #2 \right\rangle}
\newcommand{\brahket}[2]{\left| #1 \right \rangle \left \langle #2 \right|}
\newcommand{\norm}[2]{\left\Vert #1 \right\Vert_{#2}}
\newcommand{\Hnorm}[1]{\norm{#1}{\cH}}
\newcommand{\opnorm}[1]{\left\Vert #1 \right\Vert_{\mathrm{op}}}
\newcommand{\res}[2]{\mathcal{R}\left(#1, #2\right)} %resolvent operator

\theoremstyle{plain}
\newtheorem{theorem}{Theorem}[section]

\newtheorem{corollary}[theorem]{Corollary}
\newtheorem{proposition}[theorem]{Proposition}

\theoremstyle{definition}
\newtheorem{definition}[theorem]{Definition}

\newtheorem{remark}[theorem]{Remark}

\numberwithin{equation}{section}

\begin{document}

\title[]
{A note on the Krylov solvability of compact normal operators on Hilbert space}
\author[N.~Caruso]{No\`{e} Angelo Caruso}
\address[N.~Caruso]{Mathematical Institute in Opava, Silesian University in Opava \\ Na Rybnicku 626/1 \\ 74601 Opava (Czech Republic).}
\email{noe.caruso@math.slu.cz}

\makeatletter{\renewcommand*{\@makefnmark}{}
\footnotetext{Complex Analysis and Operator Theory (2023).}\makeatother}


\makeatletter{\renewcommand*{\@makefnmark}{}
\footnotetext{This version of the article has been accepted for publication, after peer review (when applicable)
but is not the Version of Record and does not reflect post-acceptance improvements, or any
corrections. The Version of Record is available online at: https://doi.org/10.1007/s11785-023-01413-0\\
The author gratefully acknowledges the support provided by the Italian National Institute of Higher Mathematics INdAM.
}\makeatother}

\date{\today}

\begin{abstract}
We analyse the Krylov solvability of inverse linear problems on Hilbert space $\cH$ where the underlying operator is compact and normal. Krylov solvability is an important feature of inverse linear problems that has profound implications in theoretical and applied numerical analysis as it is critical to understand the utility of Krylov based methods for solving inverse problems. Our results explicitly describe for the first time the Krylov subspace for such operators given any datum vector $g\in\cH$, as well as prove that all inverse linear problems are Krylov solvable provided that $g$ is in the range of such an operator. We therefore expand our knowledge of the class of Krylov solvable operators to include the normal compact operators. We close the study by proving an isomorphism between the closed Krylov subspace for a general bounded normal operator and an $L^2$-measure space based on the scalar spectral measure.
\end{abstract}

\maketitle

\section{Introduction}

The question of `Krylov solvability' of inverse linear problems is operator theoretic with deep roots in numerical applications and profound implications for the use of Krylov based methods to solve inverse linear problems. Recently this phenomenon has been studied for both bounded and unbounded operators \cite{CMN-2018_Krylov-solvability-bdd,CM-2019_ubddKrylov,CM-Nemi-unbdd-2019,CM-book-Krylov-2022}, after having received some past attention in the bounded setting \cite{Nemirovskiy-Polyak-1985,Nemirovskiy-Polyak-1985-II,Engl-Hanke-Neubauer-1996,Hanke-ConjGrad-1995,Karush-1952,Campbell-Ipsen-1996-I,Campbell-Ipsen-1996-II,CN-2018}.

In general one has a linear operator $A:\cH \to \cH$ acting on some Hilbert space $\cH$, that in many applications represents some physical law, and a datum vector $g \in \ran A$ that represents some measurable output. The inverse linear problem is formulated as
\begin{equation}\label{eq:lin_inv}
Af = g\,, \quad g \in \ran A\,,
\end{equation}
where $f \in \cH$ is a solution to the problem that in applications is often a-priori unknown. We call the problem \eqref{eq:lin_inv} \emph{solvable} as $g \in \ran A$. If $A$ is injective, we call \eqref{eq:lin_inv} \emph{well-defined}, and if additionally $A^{-1} \in \Bop$ we call the problem \eqref{eq:lin_inv} \emph{well-posed}. 

The question of `Krylov solvability' becomes relevant in applications when one attempts to solve \eqref{eq:lin_inv} by means of the very popular and celebrated family of Krylov algorithms that search for solution(s) to \eqref{eq:lin_inv} in the distinguished Krylov subspace. Therefore, one naturally wants to know whether a solution(s) $f \in \cH$ is approximable by such vectors in the Krylov subspace
\begin{equation}\label{eq:Krylov_subspace}
\mathcal{K}(A,g) := \mathrm{span}\{A^kg \,|\, k \in \N_0\}\,,
\end{equation}
or in other words, whether there exists a solution to \eqref{eq:lin_inv} $f \in \overline{\mathcal{K}(A,g)}$. Such a solution we call a \emph{Krylov solution}, and an inverse linear problem that exhibits such an occurrence we call \emph{Krylov solvable}. The practical advantage of Krylov solvable inverse linear problems is that one may construct a solution(s) to \eqref{eq:lin_inv} using the easy-to-compute vectors $g, Ag, A^2g, \dots$. Of course, the critical importance of having such knowledge a-priori is that one may decide whether a given problem is indeed a suitable candidate for treatment using a Krylov based algorithm before it is used.

The above question is already well-understood and under good control in the finite-dimensional setting, and is treated in several well-known monographs \cite{Saad-2003_IterativeMethods,Liesen-Strakos-2003,golub-vanloan-book-1996}. To a lesser extent the question of Krylov solvability has been studied in several past works in the infinite-dimensional setting for bounded operators, most of which choose to remain within a particular class of operators (e.g., positive self-adjoint), using specific Krylov algorithms (e.g., GMRES, MINRES, CG, LSQR). Recently, the problem has been studied using operator-theoretic techniques in the infinite-dimensional setting, including unbounded operators \cite{CM-2019_ubddKrylov,CM-Nemi-unbdd-2019,Olver-2009}, and has also resulted in a recent monograph on the topic \cite{CM-book-Krylov-2022}. 

In this work we choose to remain in the abstract infinite-dimensional Hilbert space setting where, innocent as the question of Krylov solvability may seem, there are several results that appear un-intuitive when coming from a finite-dimensional analysis perspective. (Indeed, several good examples of Krylov solvability, or lack thereof, of inverse linear problems may be found in \cite[E.g.~3.1]{CMN-2018_Krylov-solvability-bdd}.) 

One strategy to confront certain difficulties that naturally arise in the infinite-dimensional setting has been to identify certain classes of operators that have favourable Krylov solvability properties. In a previous study \cite{CMN-2018_Krylov-solvability-bdd} we were able to identify that the bounded self-adjoint operators always give rise to Krylov-solvable inverse linear problems, and we also identified a new class of operators that we called the `$\mathscr{K}$-class' that also always exhibit Krylov solvability. We recently expanded our study of the $\mathscr{K}$-class under the effects of perturbations in \cite{CM-krylov-perturbation-2021}. These operator classes just described in fact belong to the larger class of \emph{Krylov solvable operators}, i.e., the collection of linear operators on $\cH$ that \emph{always} admit a Krylov solution to \eqref{eq:lin_inv} given \emph{any} $g$ in the range of the operator.

Here we expand our knowledge of the class of Krylov solvable operators by proving that the compact normal operators on Hilbert space \emph{always} belong to this class. The analysis that permits us to conclude such a result is based primarily on the functional calculus for bounded operators on Banach space (see, for example \cite{Riesz-Nagy_FA-1955_Eng,Kato-perturbation,Rudin-functionalanalysis,Beauzamy-book-1988}), and the canonical decomposition of compact normal operators \cite{Beauzamy-book-1988,Kato-perturbation}. Moreover, we are able to explicitly describe the Krylov subspace in terms of the datum $g$ and the projection operators onto the eigenspaces of $A$. 

We begin this note with a preparatory theorem in Section~\ref{sec:prep_thm} (Theorem~\ref{th:1}) before moving on to the analysis specific to compact normal operators in Section~\ref{sec:compact_normal} (Propositions~\ref{prop:compactnormalKrylov}, \ref{prop:compactnormal_Ksolvable}, and \ref{prop:compactnormal_Kreduced}), and finally in Section~\ref{sec:normal_bdd} we close with two theorems for general bounded normal operators (Theorems \ref{th:normalreduced_isomorphic} and \ref{th:converse_normalreduced_isomorphic}). 

\subsection*{Notation}
Throughout this note $\cH$ denotes an abstract Hilbert space with scalar product $\scalar{\cdot}{\cdot}$ antilinear in the first argument, and norm $\Hnorm{\,}$. We use $\brahket{\varphi}{\psi}$, for $\varphi,\psi \in \cH$, to denote the rank-1 linear map $v \mapsto \scalar{\psi}{v}\varphi$ for $v \in \cH$; and $\opnorm{\,}$ to denote the standard operator norm on $\Bop$.

\section{A preparatory theorem}\label{sec:prep_thm}

Our preparatory Theorem~\ref{th:1} concerns the approximation generated by polynomials of $A \in \Bop$ of certain Riesz projections in the operator norm $\opnorm{\,}$. We shall use Theorem~\ref{th:1} in Section~\ref{sec:compact_normal} in order to analyse the structure of the Krylov subspace $\mathcal{K}(A,g)$ itself given some $g \in \cH$. We begin with a simple definition before deriving the main result.

\begin{definition}[\cite{Riesz-Nagy_FA-1955_Eng,Beauzamy-book-1988}]\label{def:admissibledomain}
An admissible domain $\mathcal{U}$ of an operator $A \in \Bop$ is a non-empty bounded open subset of the complex plane $\C$ such that the boundary $\partial \mathcal{U}$ consists of finitely many rectifiable Jordan curves contained in the resolvent $\rho(A)$ of the operator $A$ and oriented in the positive sense.
\end{definition}

\begin{theorem}\label{th:1}
Let $A \in \Bop$ with spectrum $\sigma(A)$. Suppose that $\sigma(A)$ be separated into two parts $\sigma_1$ and $\sigma_2$ such that there are admissible domains $\mathcal{U}_1$ and $\mathcal{U}_2$ containing $\sigma_1$ and $\sigma_2$ respectively; suppose further that $\overline{\mathcal{U}_1}\cap\overline{\mathcal{U}_2} = \emptyset$ and $\C \setminus (\overline{\mathcal{U}_1} \cup \overline{\mathcal{U}_2})$ is connected in $\C$. Then there exist polynomial sequences $(p_n^{(j)})_{n \in \N}$ such that $\opnorm{p_n^{(j)}(A) - P_j} \to 0$ as $n \to \infty$, where
\begin{equation}\label{eq:th_1}
P_j = \frac{1}{2\pi i} \oint_{\partial \mathcal{U}_j} \res{A}{z} \, \ud z\,,
\end{equation}
for $j \in \{1,2\}$.
\end{theorem}

\begin{proof}
Let $\mathcal{V}_1$ and $\mathcal{V}_2$ be disjoint open sets containing $\overline{\mathcal{U}_1}$ and $\overline{\mathcal{U}_2}$ respectively.
Let $f:\mathcal{V}_1\cup\mathcal{V}_2 \to [0,1]$ be a function such that $f(\mathcal{V}_1) = \{1\}$ and $f(\mathcal{V}_2) = \{0\}$. Clearly $f$ is holomorphic on $\mathcal{V}_1 \cup \mathcal{V}_2$, so using the holomorphic functional calculus (see \cite[Ch.~XI]{Riesz-Nagy_FA-1955_Eng} or \cite[Ch.~II]{Beauzamy-book-1988}) we see that
\[f(A) =  \frac{1}{2\pi i} \oint_{\partial \mathcal{U}_1 \cup \partial \mathcal{U}_2} f(z)\, \res{A}{z} \, \ud z = \frac{1}{2\pi i} \oint_{\partial \mathcal{U}_1} \res{A}{z} \, \ud z = P_1\,.\]
We know from \cite[Th.~13.7]{Rudin-realcomplexanalysis} that there exists a polynomial sequence $(p_n^{(1)})_{n \in \N}$ such that $p_n^{(1)}(z) \xrightarrow{n \to \infty} f(z)$ uniformly on $\overline{\mathcal{U}_1} \cup \overline{\mathcal{U}_2}$. As the Cauchy integral 
\[\frac{1}{2\pi i} \oint_{\partial \mathcal{U}_1 \cup \partial \mathcal{U}_2} \Big( p_n^{(1)}(z) - f(z) \Big) \res{A}{z} \, \ud z\]
is expressible as a limit of Riemann sums, we see that 
\begin{align*}
\tag{*} \opnorm{p_n^{(1)}(A) - f(A)} & \leq \frac{1}{2\pi}\sup_{z \in \overline{\mathcal{U}_1}\cup\overline{\mathcal{U}_2}} \vert p_n^{(1)}(z) - f(z) \vert \oint_{\partial \mathcal{U}_1 \cup \partial \mathcal{U}_2} \opnorm{ \res{A}{z} } \, \ud \vert z \vert \\
 & \leq \frac{l}{2\pi}\sup_{z \in \overline{\mathcal{U}_1}\cup\overline{\mathcal{U}_2}} \vert p_n^{(1)}(z) - f(z) \vert \sup_{z \in \partial\mathcal{U}_1\cup\partial\mathcal{U}_2} \opnorm{ \res{A}{z} }\,,
\end{align*}
where $l < +\infty$ is the length of the curve $\partial\mathcal{U}_1\cup\partial\mathcal{U}_2$.
The right side of (*) vanishes owing to the analyticity of $\res{A}{z}$ for all $z$ in the compact curve $\partial\mathcal{U}_1\cup\partial\mathcal{U}_2$ which guarantees  
\[\sup_{z \in \partial\mathcal{U}_1\cup\partial\mathcal{U}_2} \opnorm{ \res{A}{z} } < +\infty\,,\]
coupled with the uniform vanishing of $\vert p_n^{(1)}(z) - f(z) \vert$ on $\overline{\mathcal{U}_1} \cup \overline{\mathcal{U}_2}$ as $n \to \infty$.

As $f(A) = P_1$, we have our conclusion for $j = 1$. The proof is similar for $j = 2$ by replacing the function $f:\mathcal{V}_1\cup\mathcal{V}_2 \to [0,1]$ with any function $h:\mathcal{V}_1\cup\mathcal{V}_2 \to [0,1]$ such that $h(\mathcal{V}_2) = \{1\}$ and $h(\mathcal{V}_1) = \{0\}$.
\end{proof}

%\begin{remark}\label{rem:Theorem1_beauzamy}
%The convergence of the left-side of (*) in the proof of Theorem~\ref{th:1} is also an implication of the result \cite[Ch.~II, Prop.~2.5]{Beauzamy-book-1988}.
%\end{remark}

\section{Compact normal operators}\label{sec:compact_normal}

In this section we derive fundamental results that describe both the Krylov subspace and prove the Krylov solvability for a compact normal operator. We use this to provide a simple proof of the cyclicity of these operators on separable Hilbert space (Corollary~\ref{cor:normalcompact_cyclic}).

We use the representation given in \cite[Ch.~V]{Kato-perturbation} for a compact normal operator $A$ on Hilbert space $\cH$, namely
\begin{equation}\label{eq:comp_normalop}
A = \sum_{n \in S} \lambda_n P_n\,,
\end{equation}
where $S \subset \N_0$ is an index set with $0 \in S$, and $\lambda_n \in \C\setminus\{0\}$ for all $n \in S\setminus\{0\}$ are the distinct non-zero eigenvalues of $A$, with $\lambda_0 = 0$ not necessarily an eigenvalue of $A$. $(\lambda_n)_{n \in S}$ is a bounded sequence in $\C$ such that $\lambda_n \xrightarrow{n \to \infty} 0$ when $S$ is infinite. The $P_n$'s are mutually orthogonal projections given by 
\begin{equation}\label{eq;normal_projection}
P_n = \frac{1}{2\pi i}\oint_{\partial\mathcal{U}_n} \res{A}{z}\,\ud z\,, \quad n \neq 0\,,
\end{equation}
where $\mathcal{U}_n$ is an admissible domain that contains only the single point $\lambda_n$ from the spectrum $\sigma(A)$. $P_0$ is the orthogonal projection onto $\ker A$ (which is $\0$ when $A$ is injective), and $P_0 P_n = \0$ for all $n \in \N$. Convergence of the sum \eqref{eq:comp_normalop} occurs in the operator norm topology.

We have that the partial sums of $(P_n)_{n \in S}$ form a resolution of the identity,
\begin{equation}\label{eq:res_identity}
\1 = \sum_{n \in S} P_n\,,\text{ with } P_nP_m = \0 \quad m\neq n\,,
\end{equation}
where convergence of the sum is in the strong operator topology, so that
\begin{equation}\label{eq:res_identityP0}
\1 - P_0 = \sum_{n \in S \setminus \{0\}} P_n\,.
\end{equation}
Finally, we recall the important fact that for \emph{any} normal operator we have the relation $\ker A = \ker A^*$, ensuring that $\ran A^\perp = \ker A$.

Our first proposition reveals explicitly the structure of the Krylov subspace in a way that be more easily accessible and more meaningful for the purposes of investigating Krylov solvability and structural properties of the space than the standard definition \eqref{eq:Krylov_subspace}.
\begin{proposition}\label{prop:compactnormalKrylov}
Let $A \in \Bop$ be a compact normal operator and $g \in \cH$. Then
\begin{equation}\label{eq:prop_compactnormalKrylov}
\Kc{A}{g} = \overline{\mathspan{P_n g\,|\, n \in S}}\,.
\end{equation}
\end{proposition}

\begin{proof}
When $A = \0$, the result is trivially true as $P_0 = \1$ and $S = \{0\}$. 

Assume $A \neq \0$ so that $S \supsetneq \{0\}$, and take $n_0 \in S\setminus\{0\}$. Let $B$ be a bounded open ball about $0$ such that $\lambda_{n_0} \notin \overline{B}$ and $\partial{B} \subset \rho(A)$ (this is always possible as $\sigma(A) = \{\lambda_n\}_{n \in S}$ is discrete). There are at most finitely many points of $\sigma(A)$ outside $\overline{B}$ as $0$ is the only point of accumulation possible in the spectrum.  For each remaining point $\lambda_n \in \sigma(A) \setminus \overline{B}$ we construct bounded open balls $B_n$, containing only the respective point $\lambda_n$, with disjoint closures and also disjoint closure from $\overline{B}$ (clearly $\partial B_n \subset \rho(A)$). Let $\mathcal{U}$ be
\[\mathcal{U} = B \cup \left( \bigcup_{n \in S', n \neq n_0} B_n \right)\,,\]
where $S' \subset S$ is the finite set containing the indices of all the spectral points $\lambda_n$ not in $\overline{B}$. $\mathcal{U}$ and $B_{n_0}$ are admissible domains with mutually disjoint closures, and moreover $\C \setminus (\overline{\mathcal{U}} \cup \overline{B_{n_0}})$ is connected (indeed, $\overline{\mathcal{U}} \cup \overline{B_{n_0}}$ is the union of finitely many bounded, disjoint, closed balls).

Applying Theorem~\ref{th:1} taking $\mathcal{U}_1 = B_{n_0}$ and $\mathcal{U}_2 = \mathcal{U}$, there exists some polynomial sequence $(p^{(n_0)}_j)_{j \in \N}$ such that $\opnorm{p_j^{(n_0)}(A) - P_{n_0}} \xrightarrow{j\to\infty} 0$, and in particular $\Hnorm{p_j^{(n_0)}(A)g - P_{n_0} g} \xrightarrow{j\to\infty} 0$. Therefore $P_{n_0} g \in \Kc{A}{g}$. As $n_0 \in S \setminus \{0\}$ was arbitrary, we have following inclusion
\begin{equation}\label{eq:prop_compactnormalKrylov_1}
\overline{\mathspan{P_n g \,|\, n \in S \setminus \{0\}}} \subset \Kc{A}{g}\,.
\end{equation}
As $\sum_{n \in S \setminus \{0\}} P_n g \in \Kc{A}{g}$ from \eqref{eq:prop_compactnormalKrylov_1}, combining this with \eqref{eq:res_identityP0} implies that $(\1 - P_0)g \in \Kc{A}{g}$. The linearity of $\Kc{A}{g}$ and the fact that $g \in \Kc{A}{g}$ imply $P_0 g \in \Kc{A}{g}$. We therefore have the inclusion
\[\overline{\mathspan{P_n g \,|\, n \in S }} \subset \Kc{A}{g}\,.\]

The reverse inclusion is obtained from the fact that for any $k \in \N$,
\[A^k g = \sum_{n \in S\setminus\{0\}} \lambda_n^k P_n g\]
which implies that
\[A^k g \in \overline{\mathspan{P_n g \,|\, n \in S }}\,, \quad \forall\, k \in \N\,.\]
From equation \eqref{eq:res_identity}, $g = \sum_{n \in S} P_n g$ so that $g \in \overline{\mathspan{P_n g \,|\, n \in S }}$, and so
\[\K{A}{g} \subset \overline{\mathspan{P_n g \,|\, n \in S }}\,.\]
We conclude by taking the closure.
\end{proof}

Here we present the following corollary of Proposition~\ref{prop:compactnormalKrylov} that provides a simple exposition of the cyclicity of compact normal operators with simple eigenvalues. We recall that an operator $A \in \Bop$ is called \emph{cyclic} if there exists some $g \in \cH$ such that $\K{A}{g}$ is dense in $\cH$. Though the conclusion of Corollary~\ref{cor:normalcompact_cyclic} is already known (see, for example, \cite[Cor.~30.15]{Conway-book-2000_optheory} or \cite[Th.~1.1]{RossWogen-2004_cyclicnormal}), we choose to present it here through the lens of our explicit knowledge of the Krylov subspace provided by Proposition~\ref{prop:compactnormalKrylov}.
\begin{corollary}\label{cor:normalcompact_cyclic}
Let $A \in \Bop$ be a compact normal operator on a separable Hilbert space $\cH$ with $\dim (\ker A) \leq 1$, and $\dim (\ran P_n) = 1$ for all $n \in S \setminus \{0\}$. Then $A$ is cyclic.
\end{corollary}

\begin{proof}
We know that the $P_n$'s form a resolution of the identity and are mutually orthogonal, so there exists an orthogonal basis $\{\varphi_n\}_{n \in S}$ of $\cH$ such that $P_n = \brahket{\varphi_n}{\varphi_n}$, where $\Hnorm{\varphi_n} = 1$ for all $n \in S\setminus \{0\}$, and
\[
\Hnorm{\varphi_0} = \begin{cases}
	0\, & \text{if } \ker A = \{0\}\,,\\
	1\, & \text{if } \dim (\ker A)= 1\,.
\end{cases}
\]
Let $g = \varphi_0 + \sum_{n \in S\setminus\{0\}} \frac{1}{n} \varphi_n$ so that $g \in \cH$. By Proposition~\ref{prop:compactnormalKrylov}
\[\Kc{A}{g} = \overline{\mathspan{P_n g \,|\, n \in S}} = \overline{\mathspan{\varphi_n \,|\, n \in S}} = \cH\,.\]
\end{proof}

\begin{remark}\label{rem:conwaywogan_cyclic}
We elaborate on the comparison between the cyclicity result Corollary~\ref{cor:normalcompact_cyclic} and the more general cyclicity condition for normal operators presented in \cite[Th.~1.1]{RossWogen-2004_cyclicnormal}. Theorem~1.1 of \cite{RossWogen-2004_cyclicnormal} states that a normal operator $A \in \Bop$ is cyclic if and only if there exists a positive, finite, Borel measure $\mu$ on $\sigma(A)$ such that $A$ is unitarily equivalent to the multiplication operator $M_z :L^2(\mu) \to L^2(\mu)$ with action $f(z) \mapsto zf(z)$.

Indeed, it is enough to consider $A$ and $g \in \cH$ as given in the statement and proof respectively of Corollary~\ref{cor:normalcompact_cyclic}, and the finite positive Borel measure $\mu(\Omega) := \langle E^{(A)}(\Omega) g, g \rangle$, where $E^{(A)}$ is the unique projection valued spectral measure for $A$. $\mu$ has support exactly on $\sigma(A)$, and there exists the unitary operator $T : L^2(\mu) \to \cH$ with action $f(z) \mapsto f(A)g$. $f(A)$ is understood in terms of the spectral functional calculus
\[f(A) := \int_{\sigma(A)} f(z) \,\ud E^{(A)}(z)\,.\]
We see that $T ((M_z f)(z)) = A T f$ for all $f \in L^2(\mu)$, i.e., $M_z = T^*AT$, thus the statement of Corollary~\ref{cor:normalcompact_cyclic} satisfies the conditions of \cite[Th.~1.1]{RossWogen-2004_cyclicnormal} and therefore $A$ is cyclic.
\end{remark}

The following proposition reveals our main result, namely that compact normal operators give rise to Krylov solvable inverse linear problems, and therefore belong to a larger class of operators that \emph{always} exhibit Krylov solvability (this class contains, for example, the bounded self-adjoint and $\mathscr{K}$-class operators). The proof of this proposition is a result of the explicit knowledge of the Krylov subspace as revealed in Proposition~\ref{prop:compactnormalKrylov}.

\begin{proposition}\label{prop:compactnormal_Ksolvable}
Let $A \in \Bop$ be a compact normal operator. If $g \in \ran A$, then $Af = g$ has a unique and minimal norm Krylov solution. If in addition $A \neq \0$, then the Krylov solution is
\begin{equation}\label{eq:prop_compactnormal_Ksolution}
f_\circ = \sum_{n \in S\setminus\{0\}} \frac{1}{\lambda_n} P_n g\,.
\end{equation}
\end{proposition}

\begin{proof}
When $A = \0$ the conclusion is obvious: $g \in \ran A$ implies that $g = 0$ and therefore $\Kc{A}{g} = \{0\}$ with a solution $f_\circ = 0 \in \Kc{A}{g}$ to the inverse linear problem. Therefore we consider the case $A \neq \0$.

First we show that the vector $f_\circ$ in \eqref{eq:prop_compactnormal_Ksolution} is in $\cH$. Indeed, as $g \in \ran A$ there exists some $f \in \cH$ such that $Af = g$, and therefore
\[g = \sum_{n \in S} \lambda_n P_n f = \sum_{n \in S \setminus \{0\}} \lambda_n P_n f\,.\]
Given any $n \in S\setminus\{0\}$ owing to the mutual orthogonality of the projections $(P_n)_{n \in S}$, we get $P_n g = \lambda_n P_n f$. Therefore $P_n f = \frac{1}{\lambda_n} P_n g$ for all $n \in S\setminus \{0\}$ and
\[\Hnorm{\sum_{n \in S \setminus \{0\}} \frac{1}{\lambda_n} P_n g} = \Hnorm{\sum_{n \in S \setminus \{0\}} P_n f} = \Hnorm{(\1 - P_0)f} < + \infty\,,\]
where equation \eqref{eq:res_identityP0} is used in the last equality, so that indeed $f_\circ \in \cH$.

Next we show by direct substitution that $f_\circ$ is a solution to $Af = g$. Indeed,
\[\begin{split}
Af_\circ & = A \sum_{n \in S\setminus\{0\}} \frac{1}{\lambda_n} P_n g \\
 & = \sum_{m \in S} \lambda_m P_m \sum_{n \in S\setminus\{0\}} \frac{1}{\lambda_n} P_n g \\
  & = \sum_{m \in S\setminus\{0\}} \sum_{n \in S\setminus\{0\}} \lambda_m P_m  \frac{1}{\lambda_n} P_n g \\
  & = \sum_{n \in S\setminus\{0\}} P_n g\,.
\end{split}\]
As $g \in \ran A \perp \ker A$ this implies $P_0 g = 0$ meaning that $g = \sum_{n \in S\setminus\{0\}} P_n g$, and indeed $f_\circ$ is a solution to the inverse linear problem as claimed.

As $f_\circ \perp \ker A$, any solution $f$ to $Af = g$ must be of the form $f = f_\circ + \psi$ where $\psi \in \ker A$. Therefore, $\Hnorm{f}^2 = \Hnorm{f_\circ}^2 + \Hnorm{\psi}^2$ from which $f_\circ$ is a minimal norm solution.

From Proposition~\ref{prop:compactnormalKrylov}
\[\Kc{A}{g} = \overline{\mathspan{P_n g\,|\, n \in S}}\,,\]
so it follows $f_\circ \in \Kc{A}{g}$. This solution is unique in $\Kc{A}{g}$ owing to \cite[Prop. 3.9]{CMN-2018_Krylov-solvability-bdd}.
\end{proof}

\begin{remark}\label{rem:Ksolution_uniquenessnormal}
We recall that for \emph{any} $A \in \Bop$ normal with Krylov solvable inverse linear problem $Af = g$, $g \in \ran A$, \cite[Prop. 3.9]{CMN-2018_Krylov-solvability-bdd} states that there exists \emph{exactly one} Krylov solution. 
\end{remark}

\begin{remark}\label{rem:Ksolution_normminnormal}
The argument of the norm minimality of the Krylov solution (if it exists) from the proof of Proposition~\ref{prop:compactnormal_Ksolvable} can be extended beyond the class of compact normal operators to the whole class of bounded normal operators. Indeed, let $f_\circ \in \Kc{A}{g}$ be the Krylov solution (when it exists) to the inverse linear problem. As $\Kc{A}{g} \subset \overline{\ran A} \perp \ker A$ it follows that $f_\circ \perp \ker A$, and any solution $f$ to $Af = g$ has the form $f = f_\circ + \psi$, where $\psi \in \ker A$. Therefore, as $f_\circ \perp \psi$ we have $\Hnorm{f}^2 = \Hnorm{f_\circ}^2 + \Hnorm{\psi}^2$ from which $f_\circ$ is a minimal norm solution.
\end{remark}

\begin{remark}\label{rem:Ksolution_generalops}
The above considerations in Remarks~\ref{rem:Ksolution_uniquenessnormal} and \ref{rem:Ksolution_normminnormal} also hold for Krylov solutions (if they exist) to the inverse linear problem arising from \emph{any} operator $A \in \Bop$, provided that $\ker A \subset \ker A^*$.
\end{remark}

Our next proposition analyses an important structural property of the Krylov subspace informally known as \emph{Krylov reducibility}, that is intimately linked to the Krylov solvability properties of inverse linear problems (see \cite[Prop.~3.3]{CMN-2018_Krylov-solvability-bdd}). First we recall the appropriate definition for \emph{bounded} operators. (For the unbounded setting one may refer to \cite{CM-2019_ubddKrylov}.)

\begin{definition}[\cite{CMN-2018_Krylov-solvability-bdd}]\label{def:K-reducibility}
Let $A \in \Bop$ and $g \in \cH$. If both $\Kc{A}{g}$ and $\K{A}{g}^\perp$ are invariant under $A$, i.e.,
\begin{equation}\label{eq:def-Kreduced}
A\, \Kc{A}{g} \subset \Kc{A}{g}\,, \quad A\, \K{A}{g}^\perp \subset \K{A}{g}^\perp\,,
\end{equation}
then we say that $A$ is $\K{A}{g}$-reduced. If $A$ is $\K{A}{g}$-reduced then we also have $A^* \Kc{A}{g} \subset \Kc{A}{g}$ \cite[Lem.~2.2]{CMN-2018_Krylov-solvability-bdd}.
\end{definition}

\begin{proposition}\label{prop:compactnormal_Kreduced}
Let $A \in \Bop$ be a compact normal operator, and let $g \in \cH$. Then $A$ is $\K{A}{g}$-reduced.
\end{proposition}

\begin{proof}
The action of $A^*$ is given by
\[A^* = \sum_{n \in S} \overline{\lambda_n} P_n\,.\]
Therefore
\[A^* g =  \sum_{n \in S} \overline{\lambda_n} P_n g\,, \]
where by Proposition~\ref{prop:compactnormalKrylov} $A^*g \in \Kc{A}{g}$. We know that for bounded normal operators, $A$ is $\K{A}{g}$-reduced if and only if $A^* g \in \Kc{A}{g}$ \cite[Prop. 2.4]{CMN-2018_Krylov-solvability-bdd}. This completes the proof.
\end{proof}

\section{An isomorphism of Krylov subspaces generated by bounded normal operators}\label{sec:normal_bdd}
In this section we let $\specmeas{A}$ be the unique projection valued spectral measure for a normal operator $A \in \Bop$ (see \cite[Ch.~5]{schmu_unbdd_sa}), with associated scalar measure $\mug{A}{g}(\Omega) = \scalar{\specmeas{A}(\Omega) g}{g}$ for $g \in \cH$ and $\Omega \subset \C$ a Borel set. It is known that $\specmeas{A}$ has support only over $\sigma(A)$, and that $\mug{A}{g}$ is a positive, regular Borel measure \cite{schmu_unbdd_sa}.
\begin{theorem}\label{th:normalreduced_isomorphic}
Let $A\in\Bop$ be a normal operator and $g \in \cH$. If $A$ is $\K{A}{g}$-reduced, then there exists the isomorphism 
\begin{equation}\label{eq:th_isomorphic}
\begin{split}
L^2(\sigma(A), \mug{A}{g}) & \xrightarrow{\cong} \Kc{A}{g}\,,\\
f & \longmapsto f(A)g\,,
\end{split}
\end{equation}
where $f(A)$ is understood in terms of the spectral functional calculus
\begin{equation}\label{eq:specfunc_calc}
f(A) := \int_{\sigma(A)} f(z) \, \ud \specmeas{A}(z)\,.
\end{equation}
\end{theorem}

\begin{proof}
Let $g \in \cH$ be such that $A$ is $\K{A}{g}$-reduced. Then 
\[(A^*)^k \Kc{A}{g} \subset \Kc{A}{g}\]
for any $k \in \N_0$. This clearly implies $(A^*)^k A^n g \in \Kc{A}{g}$ for all $k \in \N_0$ and for all $n \in \N_0$. Let $q(A,A^*)$ be a polynomial in $A$ and $A^*$. As $A$ and $A^*$ commute it follows $q(A,A^*)g \in \Kc{A}{g}$.

By the Stone-Weierstrass theorem the algebra of complex coefficient polynomials in both $z$ and $\bar{z}$ on the compact set $\sigma(A)$ is dense in the space of continuous functions on $\sigma(A)$ equipped with supremum norm. As $\mug{A}{g}$ is a regular Borel measure on the compact set $\sigma(A)$, by standard approximation theorems (e.g., see \cite[Th.~3.14]{Rudin-realcomplexanalysis}) the continuous functions on $\sigma(A)$ are dense in $L^2(\sigma(A), \mug{A}{g})$. Combining the two statements above, the space of polynomials in both $z$ and $\bar{z}$ over $\sigma(A)$ is dense in $L^2(\sigma(A), \mug{A}{g})$.

Therefore, given any $f \in L^2(\sigma(A), \mug{A}{g})$ there exists a polynomial sequence $(q_n(z,\bar{z}))_{n \in \N}$ such that $q_n(z,\bar{z}) \xrightarrow[n \to \infty]{L^2} f(z)$. From the functional calculus
\[\Hnorm{q_n(A,A^*)g - f(A)g}^2 = \int_{\sigma(A)} |q_n(z,\bar{z}) - f(z)|^2 \,\ud \mug{A}{g}\,,\]
where the right side vanishes as $n \to \infty$. As $q_n(A,A^*)g \in \Kc{A}{g}$ for all $n \in \N$, we have $f(A)g \in \Kc{A}{g}$.

For any $v \in \Kc{A}{g}$ there exists a sequence of polynomials in $z$ only, $(p_n(z))_{n \in \N}$, such that $p_n(A)g \xrightarrow[n \to \infty]{\Hnorm{\,}} v$. $(p_n(A)g)_{n \in \N}$ is Cauchy in $\Hnorm{\,}$ so that
\[\Hnorm{p_n(A)g - p_m(A)g}^2 = \int_{\sigma(A)} |p_n(z) - p_m(z)|^2 \,\ud \mug{A}{g} = \norm{p_n - p_m}{L^2}^2\]
for any $m,n \in \N$, showing that $(p_n(z))_{n \in \N}$ is Cauchy in $L^2(\sigma(A),\mug{A}{g})$ as well. Thus $p_n$ converges to some $h \in L^2(\sigma(A), \mug{A}{g})$. The mapping in \eqref{eq:th_isomorphic} is therefore a linear bijection between the spaces $L^2(\sigma(A), \mug{A}{g})$ and $\Kc{A}{g}$. This concludes the proof.
\end{proof}

\begin{corollary}\label{cor:compactnormal_L2iso}
Let $A \in \Bop$ be compact and normal, and let $g \in \cH$. Then we have the isomorphism given in \eqref{eq:th_isomorphic}.
\end{corollary}

\begin{proof}
From Proposition \ref{prop:compactnormal_Kreduced} we know that $A$ is $\K{A}{g}$-reduced, and so the result follows by Theorem \ref{th:normalreduced_isomorphic}.
\end{proof}

\begin{remark}\label{rem:counterexample_normalreduced}
If $A$ is a normal operator, then it is not necessarily $\K{A}{g}$-reduced. The counter-example \cite[E.g.~3.8]{CMN-2018_Krylov-solvability-bdd} shows that the normality of an operator $A$ coupled with the Krylov solvability of the inverse linear problem is in general not enough to guarantee that $A$ is $\K{A}{g}$-reduced.
\end{remark}

Lastly, we state the converse to Theorem~\ref{th:normalreduced_isomorphic}.

\begin{theorem}\label{th:converse_normalreduced_isomorphic}
Let $A \in \Bop$ be a normal operator and $g \in \cH$. Suppose that there exists the isomorphism \eqref{eq:th_isomorphic}, \eqref{eq:specfunc_calc}. Then $A$ is $\K{A}{g}$-reduced.
\end{theorem}

\begin{proof}
The function $f:\sigma(A) \to \C$, $z \mapsto \bar{z}$ is in $L^2(\sigma(A), \mug{A}{g})$. Indeed, by the spectral integral
\[\Hnorm{A^* g}^2 = \int_{\sigma(A)} |\bar{z}|^2\, \ud \mug{A}{g} = \int_{\sigma(A)} |z|^2\, \ud \mug{A}{g} = \Hnorm{Ag}^2 < + \infty\,.\]
By hypothesis there exists the isomorphism \eqref{eq:th_isomorphic}, \eqref{eq:specfunc_calc} that subsequently implies $A^*g \in \Kc{A}{g}$ as $f(A) = A^*$. According to \cite[Prop.~2.4]{CMN-2018_Krylov-solvability-bdd} this is equivalent to the $\K{A}{g}$-reducibility of $A$. 
\end{proof}

%\section{Acknowledgements}
%The author would like to thank Prof. A Michelangeli (University of Bonn) for fruitful discussions and suggestions.

\bibliographystyle{siam}
%\bibliography{bib_ARTICLES_master.bib}

\def\cprime{$'$}

“This version of the article has been accepted for publication, after peer review (when applicable)
but is not the Version of Record and does not reflect post-acceptance improvements, or any
corrections. The Version of Record is available online at: http://dx.doi.org/[insert DOI]”

\end{document}